\documentclass[a4paper,12pt,reqno]{amsart}
\usepackage{}
\usepackage{amssymb, amsmath, amsfonts, amscd}
\usepackage{mathrsfs, latexsym, array, longtable}
\usepackage[all,ps,cmtip]{xy}
\usepackage{xcolor, bm, enumitem}
\usepackage{hyperref}

\title{A lifting principle for canonical stability indices of varieties of general type}
\author{Meng Chen, Hexu Liu}
\date{\today}
\address{\rm Meng Chen, School of Mathematical Sciences, Fudan University, Shanghai 200433, China}
\email{mchen@fudan.edu.cn}

\address{\rm Hexu Liu, Shanghai Center for Mathematical Sciences, Fudan University, Jiangwan Campus, Shanghai, 200438, China}
\email{19110840002@fudan.edu.cn}

\thanks{Project supported by NSFC for Innovative Research Groups (\#12121001), National Key Research and Development Program of China (\#2020YFA0713200) and NSFC grant \# 12071078. The first author is a member of LMNS, Fudan University}

\newcommand{\roundup}[1]{\ulcorner{#1}\urcorner}
\newcommand{\rounddown}[1]{\llcorner{#1}\lrcorner}

\newcommand\vol{\text{\rm vol}}
\newcommand{\bN}{\mathbb{N}}
\newcommand{\bZ}{\mathbb{Z}}
\newcommand{\bQ}{\mathbb{Q}}
\newcommand{\bR}{\mathbb{R}}
\newcommand{\bC}{\mathbb{C}}
\newcommand{\bP}{\mathbb{P}}
\newcommand{\fix}[1]{{\rm Fix}|#1|}

\newcommand{\lct}{{\rm lct}}
\newcommand{\Codim}{{\rm Codim}}
\newcommand{\fraction}[1]{\langle{#1}\rangle}
\newcommand{\OO}{\mathcal{O}}
\newtheorem{thm}{Theorem}[section]

\newtheorem{lemma}[thm]{Lemma}

\newtheorem{defi}[thm]{Definition}
\newtheorem{eg}[thm]{Example}
\newtheorem{cor}[thm]{Corollary}

\begin{document}
\begin{abstract}  For any integer $n>0$, the $n$-th canonical stability index $r_n$ is defined to be the smallest positive integer so that the $r_n$-canonical map $\Phi_{r_n}$ is stably birational onto its image for all smooth  projective $n$-folds of general type. We prove the lifting principle for $\{r_n\}$ as follows: $r_n$ equals to the maximum of the set of those canonical stability indices of smooth projective $(n+1)$-folds with sufficiently large canonical volumes.  Equivalently, there exists a constant $\mathfrak{V}(n)>0$ such that, for any smooth projective $n$-fold $X$ with the canonical volume $\vol(X)>{\mathfrak V}(n)$,  the pluricanonical map $\varphi_{m,X}$ is birational onto the image for all $m\geq r_{n-1}$. 
\end{abstract}
\maketitle

\pagestyle{myheadings}
\markboth{\hfill Meng Chen and Hexu Liu\hfill}{\hfill A lifting principle for canonical stability indices of varieties of general type\hfill}
\numberwithin{equation}{section}
%%%%%%%%%%%%
%\tableofcontents

\section{Introduction}

Throughout we work over the complex number field ${\mathbb C}$.\par

Within birational geometrty, an important post-MMP mission might be to work out the exact boundedness of given classes of varieties. The boundedness is usually determined by some key birational invariants such as the canonical (or anti-canonical) volume, the canonical (or anti-canonical) stability index, the Iitaka-fibration index, etc. These are crucial to explicit classifications since they directly disclose the existence of relevant moduli spaces of varieties in question.  

We start with considering a smooth projective variety $X$ of general type.  The {\em canonical stability index} of $X$ is a birational invariant, which is defined as:
$$r_s(X):=\min\{l \in \bZ_{>0}|\ \text{$\varphi_{m,X}=\Phi_{|mK_X|}$ is birational for all $m \geq l$}\}.$$ 
By the work of Hacon-McKernan (\cite{HM06}), Takayama (\cite{Tak06}) and Tsuji (\cite{Tsu06}), for any integer $n > 0$, there is a number $r_n \in \bZ_{>0}$ such that the pluricanonical map $\phi_{m}$ is birational onto its image for all $m \geq r_n$ and for all smooth projective $n$-folds of general type.  The optimal such number $r_n$ is usually referred to as the $n$-th {\it canonical stability index}. If one only considers those varieties $X$ with $p_g(X)>0$, the number $r_n^+ \in \bZ_{>0}$ is similarly defined as 
$$r_n^+:=\text{max}\{r_s(X)|\ \text{$X$\ is a smooth proj. $n$-fold of g. t. with $p_g>0$}\}.$$\par

The sequence $\{r_n\}$ (resp., $\{r_n^+\}$) of canonical stability indices is clearly increasing. Among known values, one has $r_1 = r_1^+ = 3$ and $r_2 = r_2^+ = 5$ by Bombieri (\cite{Bom73}). In dimension $3$,  one has $r_3 \leq 57$ by \cite{CC15, C21} and $14 \leq r_3^+ \leq 17$ by \cite{CHP21}. In dimensions $\geq 4$, no effective upper bound for $r_n$ is known and, however,  one has 
$r_n>2^{\sqrt{2^n}}$ by Esser-Totaro-Wang (\cite{ETW23}). 
\par 

The motivation of this article is to solve the interesting open problem, of McKernan who first put forward in Mathematics Review (see MR2339333), which is also known as \cite[Conjecture 6.1, Conjecture 6.2]{CJ17}. Our first main result is the following:

\begin{thm}\label{C-L-1} 
For any integer $n \geq 2$, there exists a constant ${\mathfrak{V}}(n)>0$ such that, for any smooth projective $n$-fold $V$ with $\vol(V) > \mathfrak{V}(n)$ (resp., $p_g(V)> \mathfrak{V}(n)$), the inequality holds
$$r_s(V)\leq r_{n-1}\  (\text{resp.}\ r_s(V)\leq r_{n-1}^+).  $$
\end{thm}

The surface case of Theorem \ref{C-L-1} was proved by Bombieri (\cite{Bom73}). The $3$-fold case was proved by Todorov (\cite{Tod07}) and the first author (\cite{C03}). The first author and Jiang (\cite{CJ17}) proved the $4$-fold case and, partially, the $5$-fold case. We noticed that Lacini (\cite{Lac23}) recently proved that, for a smooth projective $n$-fold $V$ ($n\geq 4$) with sufficiently large canonical volume, $r_s(V)\leq  \max\{r_{n-1}, (n-1)r_{n-2} + 2\}$. 

The equality in the statement ``$r_s(V)\leq r_{n-1}$ (resp. $r_{n-1}^+$)'' of Theorem
\ref{C-L-1} can attain for lots of examples. 

\begin{eg}\label{exmp1}
Let $X$ be a smooth projective $(n-1)$-fold of general type with $r_s(X) = r_{n-1}$ (resp. $r_s(X) = r_{n-1}^+$ and $p_g(X)>0$). Let $C$ be a smooth curve of genus $g\geq 2$. Take $V:=X \times C$. One sees that, as $g$ is sufficiently large, $\vol(V)$ (resp. $p_g(V)$) can be arbitrarily large and that $r_s(V) = r_{n-1}$ (resp. $r_s(V) = r_{n-1}^+$).
\end{eg}

By virtue of Example \ref{exmp1}, Theorem \ref{C-L-1} has the following equivalent form where the ``lifting principle'' gets the name:

\begin{thm}\label{C-L-1'} 
For any integer $n \geq 1$, there exists a constant ${\mathfrak{V}}(n)>0$ such that, for any number $L>{\mathfrak{V}}(n)$, one has
$$r_n=\text{max}\{r_s(X)|\ X\ \text{is a smooth proj. (n+1)-fold with}\ \vol(X)>L\}$$ 
and, respectively, 
$$r^+_n=\text{max}\{r_s(X)|\ X\ \text{is a smooth proj. (n+1)-fold with}\ p_g(X)>L\}.$$ 
\end{thm}

Theorem \ref{C-L-1} implies the following byproduct. 

\begin{cor}\label{C-L-2} For any integer $n \geq 2$, there exists a constant ${\mathfrak{V}}(n)>0$ such that, for any smooth projective $n$-fold $V$ satisfying one of the following conditions:
\begin{itemize}
\item[(1)] $h^j(\OO_V)>\mathfrak{V}(n)$ for some integer $j$ with $0<j<n$;
\item[(2)] $|\chi(\OO_V)|>{\mathfrak{V}}(n)$, 
\end{itemize}
the inequality $r_s(V)\leq r_{n-1}$ holds. 
\end{cor}

The key step in proving Theorem \ref{C-L-1} is to develop the following type of extension theorem:  

\begin{thm}\label{key}
Let $n$, $d$ be two integers with $n > d >0$ and $\mathcal{P}$ a biraionally bounded set of smooth projective varieties of dimension $d$. Then there exists a positive number $t$, depending only on $d$ and $\mathcal{P}$, such that the following property holds:
\begin{quote}
    Let $f : V \rightarrow T$ be a surjective fibration where $V$ and $T$ are smooth projective varieties, $\dim V = n$, $\dim T = n-d$ and $\vol(V)>0$. Assume that the following conditions are satisfied: 
\begin{itemize}
\item[(1)] the general fiber $X$ of $f$ is birationally equivalent to an element of $\mathcal{P}$; 
\item[(2)] there exists a positive rational number $\delta < t$ and an effective $\bQ$-divisor $\Delta \sim_\bQ \delta K_V$ such that $X$ is an irreducible non-klt center of $(V,\Delta)$. \end{itemize}
Then the restriction map
$$H^0(V,pK_V) \to H^0(X_1, pK_{X_1}) \oplus H^0(X_2, pK_{X_2})$$ is surjective for any integer $p \geq 2$ and for any two different general fibers $X_1$, $X_2$ of $f$. 
\end{quote}
\end{thm}

We briefly explain the main idea of this paper. In fact, very little information is known about the canonical stability index sequence $\{r_n\}$. It is believed to be strictly increasing, which is, however, an open question so far (see \cite[Conjecture 1.4]{Lac23}).  If one can show that $r_n\geq nr_{n-1}+2$, then the main statement of this paper holds according to a result of Lacini \cite{Lac23}. Our starting point is \cite[Question 6.6]{CJ17}, proposed by the first author and Jiang, which directly induces the solution to the main problem. Hence the key part of this paper is Section 3 where Theorem \ref{key} (a high co-dimensional version of the extension theorem) is proved. More precisely, if $f:V\to T$ is a fibration with $\dim (T)>1$, then one has the difficulty in directly using the vanishing theorem to study the restriction map
$H^0(mK_V)\to H^0(mK_X)$, where $X$ is the general fiber. Roughly speaking, our solution is to blow up $X$ and to use the exceptional divisor over $X$ as a bridge connecting the total space upstairs and X. Nakayama's Zariski decomposition, the adjunction for klt-trivial fibrations, Birkar's boundedness theorem (Theorem \ref{BAB2}) and Hacon-McKernan's extension theorem (Theorem \ref{extend}) play central roles in our argument. Details can be found in Step 1 through Step 8 in the proof of Theorem \ref{main1}.

\section{Preliminary}

\subsection{Divisors, linear systems and contractions.}
Let $D$ be an $\bR$-divisor on a variety $X$ and $P$  a prime divisor. Then denote by $\mu_P({D})$ the coefficient of $P$ in $D$. Write  $D = \sum d_i D_i$, where each $D_i$ is a prime divisor and $d_i \in \bR$. Define the {\em roundup} of $D$ to be $\roundup{D}:=\sum{\roundup{d_i}D_i}$, the {\em rounddown} of $D$ to be $\rounddown{D}:= \sum{\rounddown{d_i}D_i}$ and the {\em fractional part} of $D$ to be $\fraction{D} := \sum{\fraction{d_i}D_i}$ where $\fraction{d_i} = d_i - \rounddown{d_i}$. We say that $D$ is {\em effective} if $d_i \geq 0$ for all $i$. Given a positive real number $a$, we say that $D < a$ (resp. $D \leq a$) if $d_i < a$ (resp. $d_i \leq a$) for all $i$. \par 
Let $\{\hat{D}_j | j \in J\}$ be a set of effective $\bR$-divisors where $J$ is a set of index and write $\hat{D}_j = \sum_{i} d_{j,i} D_{j,i}$, where each $D_{j,i}$ is a prime divisor and $d_{j,i} \geq 0$ for all $j$ and $i$. Define the {\em infimum} of $\{\hat{D}_j| j \in J\}$ to be $\inf\{ \hat{D}_j |j \in J\}:= \sum_{i} (\inf_{j \in J}d_{j,i})D_{j,i}$. Assume that $J$ is an infinite subset of $\bN$, that $\liminf_{j \to \infty}d_{j,i} < +\infty$ for all $i$ and that there is only finitely many $i$ with $\liminf_{j \to \infty}d_{j,i} \ne 0$, then define $\liminf_{j \to \infty} \hat{D}_{j} := \sum_{i}(\liminf_{j \to \infty}d_{j,i}) D_{j,i}$, which is also an $\bR$-divisor on $X$. If, furthermore, $\lim_{j \to \infty} d_{j,i}$ exists for all $i$, we define $\lim_{j \to \infty} \hat{D}_{j} := \sum_{i}(\lim_{j \to \infty}d_{j,i}) D_{j,i}$.\par

Let $D$ be an integral effective divisor on a normal projective variety $X$. The {\em linear system} $|D|$ is defined as $|D|:=\{D' \geq 0| D' \sim D\}$ and the {\em $\bR$-linear system} $|D|_{\bR}$ is defined as $|D|_{\bR}:=\{D' \geq 0| D' \sim_{\bR} D\}$. The {\em fixed part} of $|D|$ is defined as  $\fix{D}:=\text{inf}\{ D' |D' \in |D|\}$. The {\em base locus} of $|D|$ is defined to be ${\rm Bs}|D| := \bigcap \{ D' | D' \in |D|\}$.
For a prime divisor $\Gamma$ on $X$, define $\sigma_{\Gamma}(D)_{\bZ} := \mu_{\Gamma}(\fix{D})$.

Let $f : X \to Z$ be a projective morphism between varieties. We say that $f$ is a {\em contraction} if $f_*\mathcal{O}_X = \mathcal{O}_Z$. In particular, a contraction is surjective and has connected fibers.

We give a generalization of \cite[Theorem 2.7]{C01}.

\begin{lemma}\label{m1}
Let $Y$ be a smooth projective variety of dimension at least 2 and $X$ be a smooth irreducible divisor on $Y$. Let $D$ and $E$ be two integral effective divisors on $Y$ such that
\begin{itemize}
\item[(1)] $X$ is not a component of $D + E$, and
\item[(2)] the image of the restriction map
$$H^0(Y,D+E) \to H^0(X, D|_X + E|_X)$$ contains the image of the inclusion
$$i:H^0(X,D|_X) \to H^0(X, D|_X + E|_X),$$ which is induced by the effective divisor $E|_X$.
\end{itemize}
Denote by $F$ the fixed part of $|D+E|$ on Y, and by $Z_X$ the fixed part of $|D|_X|$ on $X$.
Then we have the following inequality
$$F|_X \leq Z_X + E|_X.$$
\end{lemma}

\begin{proof}
By the second condition, for any effective divisor $D'_X \sim D|_X$ on $X$, there exists an effective divisor $G \sim D+E$ on $Y$ such that $G|_X = D'_X + E|_X$. By the definition of the fixed part, we know $F \leq G$ and that $F$ does not depend on the choice of $D'_X$. Thus
$$F|_X \leq \inf\{D'_X \geq 0|D'_X \sim D|_X\} + E|_X = Z_X + E|_X,$$
so we finish the proof.
\end{proof}

\subsection{Pairs and singularities.}
A {\em sub-pair} $(X,\Delta)$ consists of a normal quasi-projective variety $X$ and an $\bR$-divisor $\Delta$ such that $K_X + \Delta$ is $\bR$-Cartier. A sub-pair $(X, \Delta)$ with $\Delta \geq 0$ is called a {\em pair}. If, additionally, $\Delta$ is a $\bQ$-divisor, we call $(X,\Delta)$ a $\bQ$-pair. If a pair $(X,\Delta)$ satisfies that $X$ is smooth and that the support of $\Delta$ is of simple normal crossing, we say that $(X,\Delta)$ is a {\em log smooth pair}. \par
Let $(X,\Delta)$ be a sub-pair. Let $f:Y \to X$ be a log resolution of $(X,\Delta)$ and write $K_Y + \Gamma = f^*(K_X + \Delta)$. Then define the {\em log discrepancy} of a prime divisor $D$ on $Y$ with respect to the sub-pair $(X,\Delta)$ to be 
$$a(D,X,\Delta) := 1-\mu_D{(\Gamma)}.$$
Fix a number $\epsilon > 0$. We say $(X,\Delta)$ is {\em sub-lc} (resp., {\em sub-klt}, {\em sub-$\epsilon$-lc}) if $a(D,X,\Delta) \geq 0$ (resp., $>0$, $\geq \epsilon$) for every prime divisor $D$ over $X$. If, in addition,  $(X,\Delta)$ is a pair, we say $(X,\Delta)$ is {\em lc}  (resp., {\em klt}, {\em $\epsilon$-lc}). We say $(X, \Delta)$ is {\em terminal} if $a(D,X,\Delta) > 1$ for any prime divisor $D$ which is exceptional over $X$.

Let $(X,\Delta)$ be a sub-pair and take a log resolution $f:Y \to X$. A {\em non-klt place} is a prime divisor $D$ over $X$ with $a(D,X,\Delta) \leq 0$ and a {\em non-klt center} is the image of a non-klt place on $X$. The {\em non-klt locus} of $(X,\Delta)$ is the union of all non-klt centers of $(X,\Delta)$ which is denoted as ${\rm Nklt}(X,\Delta)$.  An irreducible component of ${\rm Nklt}(X,\Delta)$
is called an {\em irreducible non-klt center}.  An {\em lc place} is a prime divisor $D$ over $X$ with $a(D,X,\Delta) = 0$ and an {\em lc center} is the image of an lc place on $X$. We say that an lc center $G$ of $(X, \Delta)$ is a {\em pure lc center} if $(X,\Delta)$ is lc at the generic point of $G$. We say that an lc center $G$ is an {\em exceptional lc center} if there is a unique lc place over $X$ whose image is $G$ where the unique lc place is called {\em the exceptional lc place}. \par
Let $(X,\Delta)$ be an lc pair and $L \geq 0$ be an $\bR$-Cartier $\bR$-divisor. Define the {\em lc threshold} of $L$ with respect to $(X,\Delta)$ to be 
$$\lct(X,\Delta,L) := \sup\{t\in \bR|(X,\Delta + tL) {\text{ is lc}}\}.$$ 
Let $A$ be another $\bR$-Cartier $\bR$-divisor. Define the {\em lc threshold} of the $\bR$-linear system $|A|_{\bR}$ with respect to $(X,\Delta)$ to be
$$\lct(X,\Delta,|A|_{\bR}) := \inf\{\lct(X,\Delta,L)|0 \leq L \sim_\bR A\}.$$\par

\begin{lemma}\label{perturb}
Let $(X,\Delta)$ be a klt pair. Assume that
\begin{itemize} 
\item[(1)] $D\geq 0 $ is a big $\bR$-Cartier $\bR$-divisor and write $D \sim_{\bR} A + E$, where $A$ is ample and $E \geq 0$;
\item[(2)] $(X,\Delta + D)$ is not klt with an irreducible non-klt center $G$ such that $G$ is not contained in any component of $E$.
\end{itemize}
  Then,  for any $\epsilon >0$, there exist two real numbers $s$, $t$ with $0 \leq t < s  \leq 1$, $s + t < 1 + \epsilon$ and an effective divisor $D' \sim_\bR (s +t) D$ such that $G$ is a pure and exceptional lc center of $(X,\Delta + D')$.\par
If, in addition, $(X,\Delta)$ is a $\bQ$-pair and $D$ is a $\bQ$-divisor, we can take $s,\ t \in \bQ$ and $D'$ a $\bQ$-divisor.
\end{lemma}

\begin{proof}
Take a log resolution $f:Y \to X$ of $(X, \Delta + D)$ and write
\begin{eqnarray*}
 f^*(K_X + \Delta)&=& K_Y + \tilde{\Delta} + \sum{a_i E_i}, \\ 
f^*(D) &=& \tilde{D} + \sum{b_i E_i}, 
\end{eqnarray*}
where $\tilde{\Delta},\tilde{D}$ are the biratianal transform of $\Delta, D$, respectively, on $Y$ and all $E_i$ are $f$-exceptional divisors on $Y$.\par
Then we take $F:= \sum{c_i E_i}$ to be an $\bR$-divisor on $Y$ supported on the exceptional locus with sufficiently small positive coefficients such that $H:=f^*A - F$ is ample on $Y$. Take a general $ H_1 \sim_\bR H$ and define $A_1 := f_* H_1$ and $D_1 = A_1 + E$. Then $D_1 \sim_\bR D$. By the negativity lemma \cite[Lemma 3.3]{Bir12} and the direct computation, one has
$$f^*(D_1) = H_1 + \sum{c_i E_i} + E',$$
where $E':=f^*E$. We need to mention that $H_1$ does not contain any exceptional divisors since it is general and that $E'$ does not contain any exceptional divisors mapping onto $G$ since $E$ does not contain $G$.\par
Next, we pick a sufficiently small real number $t \in [0,\epsilon)$, and let $s$ be the largest number such that $(X, \Delta + tD_1 + sD)$ is lc at the generic point of $G$. It follows that $s\leq 1$. We have
$$f^*(K_X + \Delta + tD_1 + sD) = K_Y + \tilde{\Delta} + tH_1 + tE' + s \tilde{D} + \sum{(a_i + t c_i + s b_i) E_i}.$$
After appropriately perturbing the coefficients $c_i$, we may and do assume that there exists a unique $E_j$ mapping onto $G$ such that the coefficient $a_j + t c_j + s b_j$ is equal to $1$. In other words, $E_i$ is the unique lc place of $(X, \Delta + tD_1 + sD)$ whose image on $X$ is $G$. Finally, let $D':=tD_1 + sD$, which satisfies all the requirements.\par
For the $\bQ$-divisor case, we can assume that  $A$, $E$, $F$, $H_1$ are $\bQ$-divisors and take $t$ to be a rational number. It follows that $s$ is also a rational number. Thus $D'$ is a $\bQ$-divisor. 
\end{proof}

We need to use the following useful lemma in this paper. 

\begin{lemma}\label{m1.5} (cf. \cite[Proposition 2.36(1)]{Kol98})    Let $(X,\Delta)$ be a sub-klt sub-pair. Then there is a log resolution $\nu : Y \to X$ of $(X, \Delta)$ such that, if we write $$K_Y + \Gamma  - F = \nu^*(K_X + \Delta)$$
    where $\Gamma$ and $F$ are effective with no common components, any two components of $\Gamma$ do not intersect. In particular, $(Y,\Gamma - F)$ is terminal. 
\end{lemma}

\subsection{The canonical volume}\label{vol}

Let $X$ be a normal projective $n$-fold with $\bQ$-factorial and at worst canonical singularities. Then we define the {\em canonical volume} of $X$ to be:
$$\vol(X) = \varlimsup\limits_{m\to\infty}\frac{h^0(X,mK_X)}{m^n/n!}.$$
The {\em geometric genus} is defined  as $p_g(X):=h^0(X,K_X)$. Note that the canonical volume and the geometric genus are both birational invariants.\par

Here we recall a useful inequality between the canonical volume and the geometric genus: 

\begin{lemma}\label{Noether} 
(\cite[Theorem 5.1]{CJ17}) Let $n>0$ be any integer. There exist positive numbers $a_n$ and $b_n$ depending only on the dimension $n$ such that, for all smooth projective n-folds $X$ of general type, the following inequality holds:
$$\vol(X) \geq a_n p_g(X) - b_n.$$
\end{lemma}

Particularly, whenever $p_g(X)$ is sufficiently large, so is $\vol(X)$.

\subsection{Nakayama-Zariski decomposition.}

We mainly refer to \cite[Chapter III]{N04} for the following definitions.\par

Let $X$ be a complex smooth projective variety. For a big $\bR$-divisor $D$ on $X$ and any prime divisor $\Gamma$, define 
$$\sigma_{\Gamma}(D) := \inf\{\mu_{\Gamma}(D')| 0 \leq D' \sim_\bR D\}.$$
Then $\sigma_{\Gamma}$ only depends on the numerical class of $D$ and it is a continuous function on the big cone. If $D$ is pseudo-effective, define $\sigma_{\Gamma}(D) := \lim_{\epsilon \to 0}\sigma_{\Gamma}(D + \epsilon A)$, where $A$ is any ample divisor. Note that this definition does not depend on the choice of $A$.
Then define
$N_\sigma(D) := \sum_{\Gamma}\sigma_{\Gamma}(D)\Gamma$
and $P_\sigma(D):= D -N_\sigma(D).$
We call the decomposition $D = P_\sigma(D) + N_\sigma(D)$ the {\em $\sigma$-decomposition}. If $P_\sigma(D)$ is nef, we say that $D$ {\em admits a Zariski-decomposition}.\par
Let $f: X \to Y$ be a projective surjective morphism between smooth varieties and $D$ be a $f$-pseudo-effective $\bR$-divisor on $X$. Then one may similarly define the relative version of $\sigma$-decomposition. In fact, on $X$, assume that $D$ is $f$-big and $\Gamma$ is a prime divisor. Then define 
$$\sigma_{\Gamma}(D;X/Y) := \inf\{\mu_{\Gamma}(D')| 0 \leq D' \sim_\bR D \text{ over $Y$}\}.$$

If $D$ is $f$-pseudo-effective, define $\sigma_{\Gamma}(D;X/Y):= \lim_{\epsilon \to 0}\sigma_{\Gamma}(D + \epsilon A;X/Y)$ for a relatively ample divisor $A$. Again it is seen that the definition does not depend on the choice of $A$. If we have $\sigma_{\Gamma}(D;X/Y) < + \infty$ for any prime divisor $\Gamma$ on $X$, then we define that
$$N_\sigma(D;X/Y) := \sum_{\Gamma}\sigma_{\Gamma}(D;X/Y)\Gamma,$$
and that $P_\sigma(D;X/Y):= D -N_\sigma(D;X/Y)$.

We list some basic properties about the $\sigma$-decomposition. 
\begin{lemma}\label{m2}
Let $X$ be a smooth projective variety, $D$ be a pseudo-effective $\bR$-divisor on $X$ and $f:X \to Z$ be a projective surjective morphism. Then
\begin{enumerate}
\item[(1)] $N_\sigma(D) \geq N_\sigma(D;X/Z)$ always holds. 
\item[(2)] If $D$ is integral and effective, then
$$H^0(X,D) = H^0(X, D - \roundup{N_\sigma(D)}).$$
\item[(3)] If $D$ is big, $N_\sigma(D) = \lim_{m \to +\infty}{\frac{1}{m}{\rm Fix}|\rounddown{mD}|}$.

\item[(4)] There is an ample divisor $A_0$, such that for any ample divisor $A$ with $A-A_0$ ample, we have $$N_\sigma(D) = \lim_{m \to +\infty}{\frac{1}{m}{\rm Fix}|\rounddown{mD} + A|}.$$
\item[(5)] If $D$ is big and the ring $R(X,D):= \oplus_{m=0}^\infty H^0(X,\rounddown{mD})$ is a finitely generated $\bC$-algebra, then there exists a birational morphism $g : Y \to X$ from a smooth projective variety such that $P_\sigma(g^*D)$ is a semi-ample $\bQ$-divisor. In particular, $g^*D$ admits a Zariski-decomposition.
\end{enumerate}
\end{lemma}

\begin{proof} Statements (1) and (2) are obtained directly from definitions. Statements  (3) and (5) are due to  \cite[III, Remark 1.17]{N04}.\par
Now we prove Statement (4), for which we mainly refer to \cite[V, Corollary 1.7(1)]{N04}. Let $W:= {\rm NBs}(P_\sigma(D))$ which is defined as \cite[III, Definition 2.6]{N04}. Then, by \cite[V, Page 168]{N04}, $W$ is a countable union of subvarieties of codimension at least $2$. In fact, by \cite[V, Theorem 1.3]{N04}, there is an ample divisor $A_0$ such that for any ample divisor $A$ with $A-A_0$ ample and, for any $x \in X \setminus W$ and any $m \geq 0$, we have $x \not\in {\rm Bs}|\rounddown{mP_\sigma(D)} + A|$. We can further assume that $A$ has no common components with $D$ and $N_{\sigma}(D)$.\par
Then we only need to prove that, for any prime divisor $\Gamma$, 
$$ \lim_{m \to + \infty}{\frac{1}{m}\sigma_{\Gamma}(\rounddown{mD} + A)_{\bZ}} = \sigma_{\Gamma}(D).$$

Since for any prime divisor $\Gamma$, $\Gamma \setminus W$ is non-empty, we can take a closed point $y \in \Gamma \setminus W$. Since  $y \not\in {\rm Bs}|\rounddown{mP_\sigma(D)} + A|$, we see that
$$\sigma_{\Gamma}(\rounddown{mP_\sigma(D)} + A)_{\bZ} = 0.$$
Thus one has
\begin{align*}
\sigma_{\Gamma}(\rounddown{mD} + A)_{\bZ} &\leq \mu_{\Gamma}(\rounddown{mD} - \rounddown{mP_\sigma(D)}) \\
&\leq \mu_{\Gamma}(\roundup{mN_{\sigma}(D)}) \leq m\sigma_{\Gamma}(D) + 1.
\end{align*}
Therefore 
$$\varlimsup_{m \to \infty} \frac{1}{m}\sigma_{\Gamma}(\rounddown{mD} + A)_{\bZ} \leq \sigma_{\Gamma}(D).$$

On the other hand, for any $k > 0$, any sufficienly large real number $m$, and any $\Gamma$ with $\sigma_{\Gamma}(D) > 0$,
$$\frac{1}{mk}\sigma_{\Gamma}(\rounddown{mkD} + A)_{\bZ} \geq \frac{1}{mk} \sigma_{\Gamma}(\rounddown{m(kD + A)})_\bZ.$$
Since $kD + A$ is big, Statement (3) implies that, for any $k > 0$, 
\begin{align*}
    \varliminf\limits_{m \to \infty}\frac{1}{m}\sigma_{\Gamma}(\rounddown{mD} + A)_{\bZ} &= \varliminf\limits_{m \to \infty}\frac{1}{mk}\sigma_{\Gamma}(\rounddown{mkD} + A)_{\bZ} \\
    &\geq \frac{1}{k} \varliminf\limits_{m \to \infty}{\frac{1}{m}\sigma_{\Gamma}(\rounddown{m(kD + A)})_{\bZ}} \\
    &= \frac{1}{k}\sigma_{\Gamma}(kD + A) = \sigma_{\Gamma}(D + (1/k)A).
\end{align*}
Then, taking the limit as $k\to \infty$, we get
$$\varliminf\limits_{m \to \infty}\frac{1}{m}\sigma_{\Gamma}(\rounddown{mD} + A)_{\bZ} \geq \sigma_{\Gamma}(D).$$
The lemma is proved. 
\end{proof}

\subsection{Vertical divisors}
Let $f:X \to Y$ be a contraction between normal projective varieties and $D$ be an $\bR$-divisor on $X$. We say that $D$ is {\em horizontal} over $Y$ or {\em $f$-horizontal} if the support of $f(D)$ dominates $Y$, otherwise we say $D$ is {\em vertical} over $Y$ or {\em $f$-vertical}. We say that $D$ is {\em very exceptional} over $Y$ if $D$
satisfies the following conditions:
\begin{enumerate}
    \item [(1)] $D$ is vertical over $Y$;
    \item [(2)] For any prime component $P$ of $D$, either $f(P)$ has codimension at least $2$ in $Y$, or there exists another prime divisor $Q$ on $X$ such that $f(Q) = f(P)$, but $Q$ is not contained in the support of $D$.
\end{enumerate}
By definition, whenever $\Codim f(D) \geq 2$, $D$ is automatically very exceptional.

\begin{lemma}\label{m3} (cf. \cite[III, Proposition 5.7]{N04}) Let $f:X \to Y$ be a contraction between smooth projective varieties and $D$ be an effective $\bR$-divisor on $X$. Assume that $D$ is very exceptional over $Y$. Then $D = N_\sigma(D;X/Y)$.
\end{lemma}

\begin{lemma}\label{m4}
Let $f:X \to Z$ be a contraction between smooth projective varieties. Let $Q$ be an effective $\bR$-divisor on $X$ and be vertical over $Z$. Assume that $A$ is a very ample divisor on $Z$ such that $f^*(aA) - Q$ is pseudo-effective for a positive number $a$. 
Then there exists an effective divisor $\Delta \sim_{\bR} 2aA$ such that $f^*(\Delta) \geq Q$.
\end{lemma}

\begin{proof}
By the flattening theorem (see \cite[Chapter 3]{Ray72},\cite{H75},\cite[Theorem 3.3]{Vil06}), we can take a birational morphism $\sigma: Z' \to Z$ that flattens $f:X \to Z$ such that $Z'$ is smooth.  Denote by $U$ the normalization of the main component of $X \times_{Z} Z'$. Then the induced map $\pi: U \to X$ is birational and $f': U \to Z$ is flat. Let $\nu: X' \to U$ be a resolution of $U$, $\rho:= \pi \circ \nu : X' \to X$ and $g:= f' \circ \nu: X' \to Z'$.
$$ 
\xymatrix{
X' \ar[r]^\nu \ar[dr]_{g} \ar@/^1.3pc/[rr]^{\rho}&	U \ar[d]_{f'} \ar[r]^{\pi} & X \ar[d]^{f}\\
 &Z' \ar[r]^{\sigma} & Z
}$$

Since $f'$ is flat, we see that for any $f'$-vertical divisor $D$ on $U$, $f'(D)$ is a divisor on $Z'$. 

Let $P:=\inf\{S \geq 0 \text{ on $Z'$} | f'^*S \geq \pi^*Q\}$. Then $P$ is well-defined. Set $E:= f'^*P - \pi^*Q$. Then $E \geq 0$ and $E$ is very exceptional over $Z'$. Indeed, for any component $D$ of $E$, if we set $S:=f'(D)$ on $Z'$ and there is a prime divisor $D'$ on $U$ such that $f'(D') = S$ and the coefficient of $D'$ in $f'^*P - \pi^*Q$ is zero by the definition of $P$. 

Let $E' := \nu^*E$ on $X'$. Then $E'$ is also very exceptional over $Z'$. Hence $N_\sigma(E';X'/Z') = E'$ by Lemma \ref{m3}. \par

Since $ g^*(\sigma^*(aA) - P) + E' =  \rho^*(f^*(aA) - Q )$ is pseudo-effective, 
\begin{align*}
N_\sigma(g^*(\sigma^*(aA) - P) + E') &\geq N_\sigma(g^*(\sigma^*(aA) - P) + E';X'/Z') \\
&= N_\sigma(E';X'/Z') = E'.
\end{align*}
Then $g^*(\sigma^*(aA) - P) \geq P_\sigma(g^*(\sigma^*(aA) - P) + E')$, which is pseudo-effective. Hence $\sigma^*(aA) - P$ is pseudo-effective by \cite[Chapter II, Lemma 5.6(2)]{N04}. Since $\sigma^*(A)$ is nef and big, there exists an effective divisor $\Theta \sim_\bR \sigma^*(2aA) - P$ and then take $\Delta = \sigma_*(\Theta + P)$, which satisfies our requirement. Indeed, 
since $\Theta + P$ is relatively trivial over $Z$, $\sigma^*(\Delta) = \Theta + P$ by the negativity lemma \cite[Theorem 3.3]{Bir12}.
So it follows that $\rho^*(f^*\Delta - Q) = g^*\Theta + E' \geq 0$. Since $\rho$ is birational, we have $f^*(\Delta) - Q \geq 0 $.
\end{proof}
\subsection{B-divisors} We recall the definition of b-divisors introduced by Shokurov. For details, one may refer to \cite{Amb04} and \cite{Fuj12} for instance. \par
Let $\mathbb{K}$ be $\bZ$, $\bQ$ or $\bR$. Let $X$ be a normal projecitve variety. A {\em  b-$\mathbb{K}$-divisor} $\mathbf{D}$ of $X$ consists of a family of $\mathbb{K}$-divisors $\{D_{X'}\}$ indexed by all the birational models $X'$ of $X$ which satisfies that, for any birational morphism $\nu : X'' \to X'$ between two birational models $X'$,$X''$ of $X$, $\nu_*D_{X''} = D_{X'}$ holds.  \par
Let $\mathbf{K}$ consist of $\{K_{X'}\}$ for all birational models $X'$, where $K_{X'}$ is the canonical divisor defined by a top rational differential form $(\omega)$ of $X$ such that $\nu_*K_{X''} = K_{X'}$ for any birational morphism $\nu: X'' \to X'$. We call $\mathbf{K}$ the {\em canonical b-divisor} of $X$.\par
Let $(X,B)$ be a sub-pair. We recall the definition of the {\em discrepancy b-divisor} $\mathbf{A}(X,B):=\{A_{X'}\}$ as in \cite[2.3]{FG14} and  \cite[Definition 3.6]{Fuj12}. For any birational morphism $\nu : X' \to X$ from a normal variety $X'$, we define $$A_{X'}:=K_{X'} - \nu^*(K_X + B).$$
Then we define $\OO_{X}(\roundup{\mathbf{A}(X,B)}) :=\tilde{\nu}_*\OO_{Y}(\roundup{A_{Y}})$ where $\tilde{\nu}: Y \to X$ is any log resolution of $(X,B)$.
\subsection{Adjunction.}{\label{adjunction}} We recall the adjunction for fiber spaces, for which we mainly refer to \cite{Amb99},\cite{Amb04} and \cite{FG14}.\par
\begin{defi}
{\em A {\em klt-trivial fibration} $f: (X,B) \to Z$ consists of a contraction $f:X \to Z$ between normal projective varieties and a sub-pair $(X,B)$, which satisfies the following conditions:
\begin{enumerate}
    \item[(1)] $(X,B)$ is sub-klt over the generic point of $Z$;
    \item[(2)] ${\rm rank}\ f_*\OO_{X}({\roundup{\mathbf{A}(X,B)}}) =1$;
    \item[(3)] $K_X + B \sim_{\bR,f} 0$ over $Z$.
\end{enumerate}}
\end{defi}
We will do adjunction for a klt-trivial fibration $f: (X,B) \to Z$. For any prime divisor $P_i$ on $Z$, define
$$t_i := \sup\{a\in \bR|(X,B + af^*P_i)\text{ is lc over the generic point of }P_i\}.$$
Then define $B_Z := \sum(1-t_i)P_i$, where the sum is taken over all prime divisors on $Z$. Then it is seen that $B_Z$ is a divisor and we call $B_Z$ the {\em discriminant part} of $(X,B)$ on $Z$ (see \cite[3.4]{FG14}). Then there exists an $\bR$-divisor $M_Z$ such that 
$K_X + B \sim_\bR f^*(K_Z + B_Z + M_Z).$ We call $M_Z$ the {\em moduli part} of $(X,B)$ on $Z$. Note that $M_Z$ is only determined up to $\bR$-equivalence.\par

Here we list several properties about adjunction for fiber spaces.

\begin{enumerate}
\item[(i)] Let $\nu:X' \to X$ be a birational morphism and $K_{X'} + B' = \nu^*(K_X + B)$ be the crepant pullback. One may do adjunction for $f\circ\nu : (X', B') \to Z$. Then the discriminant part of $(X', B')$ on $Z$ is the same as $B_Z$. As being observed in \cite[Remark 3.1]{Amb99}, while doing adjunction, we are free to take a crepant model.

\item[(ii)] Let $\sigma:Z' \to Z$ be a birational morphism from a normal projective variety $Z'$. We may take a crepant model $(X', B')$ of $(X, B)$ and assume that $f$ induces a contraction $f': X' \to Z'$. If we do adjunction for $f': (X',B') \to Z'$, we can define the discriminant part $B_{Z'}$ and the moduli part $M_{Z'}$ of $(X',B')$ on $Z'$. Then we have $\sigma_*B_{Z'} = B_Z$. Although the moduli part is only defined up to $\bR$-linear equivalence, we can take a compatible moduli part $M_{Z'}$ so that
$\sigma_*M_{Z'} = M_Z$. Then we have
$$K_{Z'} + B_{Z'} + M_{Z'} = \sigma^*(K_Z + B_Z + M_Z).$$
Moreover, let $A$ be an $\bR$-divisor on $Z$ and set $A':= \sigma^*(A)$ on $Z'$. If we do adjunction for $f: (X, B + f^*A) \to Z$, by direct computations, we will get the discriminant part $B_Z + A$ and the moduli part $M_{Z}$. Similarly, if we do adjunction for $f' : (X', B' + {f'}^*(A')) \to Z'$,  the discriminant part and the moduli part will be $B_{Z'} + A'$ and $M_{Z'}$, respectively.
\end{enumerate} 

\begin{defi} {\em 
    Let $f : (X, B) \to Z$ be a klt-trivial fibration. Let $\sigma : Z' \to Z$ be a birational morphism from a normal projective variety. We say that $Z'\to Z$ is an {\em excellent base extension with regard to $f : (X, B) \to Z$} if the following property holds:  
    \begin{quote} 
        Let $(X', B')$ be a crepant model of $(X,B)$ and assume that $f$ induces a contraction $f' : X' \to Z'$. If we do adjunction for $f': (X',B') \to Z'$, the moduli part $M_{Z'}$ is nef and, furthermore, for any birational morphism $Z'' \to Z'$ from a normal projective variety  $Z''$, the respective moduli part $M_{Z''}$ is the pullback of $M_{Z'}$ (and, hence, is nef as well). 
    \end{quote}}
\end{defi}

\begin{lemma}\label{m4.5}(see \cite[Theorem 0.2]{Amb04})
    Let $f : (X, B) \to Z$ be a klt-trivial fibration. If we assume that $(X,B)$ is a $\bQ$-sub-pair, then there exists an excellent base extension with regard to $f : (X, B) \to Z$. 
\end{lemma}

The next lemma with regard to singularities of adjunction is useful to our argument. 

\begin{lemma}\label{m5}
Let $f:(X,B) \to Z$ be a klt-trivial fibration. Assume that there exists an excellent base extension  $\sigma: Z' \to Z$ with regard to $f: (X,B) \to Z$. Let $(X', B')$ be a crepant model of $(X,B)$ and assume that $f$ induces a contraction $f' : X' \to Z'$ and denote by $B_{Z'}$ and $M_{Z'}$ the discriminant part and the moduli part on $Z'$ of $f' : (X',B') \to Z'$. Then $(X,B)$ is sub-klt provided that $(Z',B_{Z'})$ is sub-klt. 
\end{lemma}

\begin{proof}

Let $(X',B')$ be a crepant model of $(X,B)$ such that the induced map $f' : X' \dashrightarrow Z'$ from $f$ is a morphism. Denote by $\nu: X' \to X$ the birational morphism. \par

Suppose that $(X,B)$ is not sub-klt. Then there is a prime divisor $D$ over $X$ such that $a(D,X,B) \leq 0$. We will construct a higher model $f'':X''\to Z''$ over $f':X'\to Z'$ as follows. \par

{}First, take a birational model $X''$ of $X'$ such that $D$ is a divisor on $X''$. Denote by $\pi : X'' \to X$ the birational morphism. Since $(X,B)$ is sub-klt over the generic point of $Z$, we see that $D$ is vertical over $Z$. Next, take a birational morphism $\rho : Z'' \to Z$ that flattens $f \circ \pi : X'' \to Z$. Finally, after possibly blowing-up $Z''$ while replacing $X''$ with a higher birational model and replacing $D$ with its birational transform, we may and do assume that
\begin{enumerate}
    \item[(1)] There are birational morphisms $\tau : Z'' \to Z'$ and $\xi : X'' \to X'$ where $X''$ and $Z''$ are both smooth projective;
    \item[(2)] the induced map $f'': X'' \dashrightarrow Z''$ is a morphism;
    \item[(3)] $f''(D)$ is a prime divisor $S$ on $Z''$.
\end{enumerate}
$$
\xymatrix{
 X'' \ar[r]^{\xi} \ar[d]^{f''}& X' \ar[r]^{\nu} \ar[d]^{f'}	& X \ar[d]^{f} \\
 Z'' \ar[r]^{\tau} &	 Z'\ar[r]^{\sigma}  & Z
}
$$

Let $K_{X''} + B'' = \xi^*\circ\nu^*(K_X + B)$ be the crepant pullback and let $B_{Z''}$ and $M_{Z''}$ be the respective discriminant part and the moduli part on $Z''$ for $(X'',B'')$. We have $\mu_D(B'')\geq 1$ since $a(D,X,B) \leq 0$. Thus, by the definition of discriminant part, one has  $\mu_{S}(B_{Z''}) \geq 1$, which implies that $(Z'',B_{Z''})$ is not sub-klt. On the other hand, since $Z'\to Z$ is an excellent base extension, one has $K_{Z''} + B_{Z''} = \tau^*(K_{Z'} + B_{Z'})$. By assumption that $(Z',B_{Z'})$ is sub-klt, we see that $(Z'',B_{Z''})$ is also sub-klt, a contradiction.
\end{proof}

The last part of this subsection is devoted to exploring adjunction for lc centers.\par

Let $(X,B)$ be a pair with a pure and exceptional lc center $G \subset X$ and denote by $\nu:F \to G$ the normalization. Let $f: Y \to X$ be a log resolution of $(X,B)$ and denote by $S$ the unique lc place on $Y$ whose image on $X$ is $G$. Write $K_Y + B_Y = f^*(K_X + B)$, and define $B_S := (B_Y - S)|_S$. Let $g: S \to F$ be the unique morphism induced by the restriction of $f$ to $S$. By \cite[Lemma 4.1]{Amb99}, one has 
\begin{enumerate}
    \item[(iv)] $g$ is a contraction;
    \item[(v)] $(K_S + B_S) \sim_{\bR,g} 0$ over $F$;
    \item[(vi)] $(S,B_S)$ is sub-klt over the generic point of $F$;
    \item[(vii)] $g_*\OO_{S}{(\roundup{-B_S})_{\eta_F}} = \OO_{F,\eta_F}$, where $\eta_{F}$ is the generic point of $F$.
\end{enumerate}
 Thus $g: (S,B_S) \to F$ is a klt trivial fibration and we can do adjunction to get the discriminant part $B_{F}$ and the moduli part $M_{F}$ of $(S,B_S)$ on $F$. Thus we can write $$K_S+B_S\sim_\bR f^*(K_F+B_F+M_F)$$
which directly implies 
$$(K_X + B)|_F:=j^*(K_X+B) \sim_\bR K_F + B_F + M_F$$
where $j: F\to G\hookrightarrow X$ is the composition. Hence we have:

(viii) Since $B \geq 0$, by \cite[Lemma 4.2 (3)]{Amb99}, one has $B_F \geq 0$.

\subsection{The Chow variety.}
Given a polarized projective variety, one would like to parametrize all its sub-varieties with a fixed dimension and a fixed degree. So we need to recall the Chow variety which was constructed in \cite[Chapter I]{K96}.\par
Let $X$ be a polarized projective variety. Define the following moduli functor $\mathcal{C}how_{n,d}(X)$ by 
$$\mathcal{C}how_{n,d}(X)(Z):=\left\{
\begin{aligned}
&\text{Families of nonnegative algebraic cycles of } \\
&\text{dimension $n$ and degree $d$ of } X \times Z/Z
\end{aligned}
\right\}.
$$
Then, by \cite[Chapter I, Theorem 3.21]{K96}, the functor $\mathcal{C}how_{n,d}(X)$ is represented by a universal family
$$p: {\rm Univ}_{n,d}(X) \to {\rm Chow}_{n,d}(X),$$
where ${\rm Chow}_{n,d}(X)$ is projective.

\subsection{Bounded families.}
We say that a set $\mathscr{X}$ of varieties is {\em bounded} (resp., {\em birationally bounded}) if there is a projective morphism $Z \to S$ such that $S$ is of finite type and that, for any $X \in \mathscr{X}$, there exists a closed point $s \in S$ and an isomorphism $X \to Z_s$ (resp., {a birational map $X \dashrightarrow Z_s$}), where $Z_s$ is the fiber of $s$. \par

We need the following lemma to slightly relax the condition of boundedness to birational boundedness.   

\begin{lemma}\label{m7}
Let $d$ be a positive integer and $M$ be a positive real number. 
Then there exists a bounded family $\mathcal{P}$ consisting of smooth projective varieties and depending only on $d,M$ such that the following property holds:
\begin{quote}
Assume that 
	\begin{enumerate}
		\item[(1)] $f : V \to T$ is a contraction between smooth projective varieties with $\dim V - \dim T = d$, and
		\item[(2)] any general fiber $X$ of $f : V \to T$ is of general type with $\vol(X) \leq M$.
	\end{enumerate}
	Then there exists the following commutative diagram
	$$
	\xymatrix{
	V' \ar[dr]^{g} \ar@{-->}[rr]^{\psi}& &V \ar[dl]_{f} \\
	 & T
	}
	$$
	satisfying: 
	\begin{enumerate}
		\item[(i)] $\psi : V' \dashrightarrow V$ is a birational map from a smooth projective variety $V'$;
		\item[(ii)] any general fiber $X'$ of $g$ belongs to the bounded family $\mathcal{P}$.
	\end{enumerate}
	If we assume further that 
	\begin{enumerate}
		\item[(3)] $V$ is of general type and for any general fiber $X$ of $f: V \to T$, there is an effective $\bQ$-divisor $\Delta \sim_\bQ \delta K_V$ on $V$ such that $X$ is an irreducible non-klt center of $(V,\Delta)$,
	\end{enumerate}
	then we can further require that
	\begin{enumerate}
		\item [(iii)] for a general fiber $X'$ of $g : V' \to T$, there is an effective $\bQ$-divisor $\Delta' \sim_\bQ \delta K_{V'}$ on $V'$ such that $X'$ is an irreducible non-klt center of $(V',\Delta')$.
	\end{enumerate}
\end{quote}
\end{lemma}

\begin{proof} The proof is organized through several steps.\\

{\em Step 1.} The construction of $\mathcal{P}$ using the Chow variety. 

By the boundedness theorem in \cite{HM06},\cite{Tak06} and \cite{Tsu06}, there exists an integer $r_d$ depending only on $d$ such that $|r_dK_X|$ induces a birational map onto its image for any general fiber $X$.\par
Set $S:= {\rm Chow}_{d,\leq {r_d}^d M}(\bP^{N})$, the Chow variety of $d$-dimensional varieties in $\bP^{N}$ with degree $\leq {r_d}^d M$, where $N:= \rounddown{{r_d}^{d}M} + d$.  Then $S$ is projective and of finite type. Let $p:\mathscr{X} \to S$ be the universal family corresponding to the Chow variety. \par
First, take a stratification of $p : \mathscr{X} \to S$ as
$$ \coprod_{\rm finite} S_i \to S$$ 
such that each $S_i$ is smooth and that each $p_i$ is flat, where $p_i : \mathscr{X}_i \to S_i$ is the corresponding family. \par  
Then, for each $\mathscr{X}_i$, take a resolution $\mathscr{Y}_i \to \mathscr{X}_i$ by finitely many times of blowups of singularities. After possibly replacing $S_i$ with a further stratification, we can assume that each exceptional locus dominates certain $S_i$. %Indeed, this can be done as follows, if any center of a step of blowup does not dominate $S_i$, then we can "remove" the image of this center under $p_i$ from $S_i$ to make a new stratification.\par 
Therefore we get a smooth bounded family $q=\coprod q_i : \coprod \mathscr{Y}_i \to \coprod S_i$, and denote by $\mathcal{P}$ this family.\\

{\em Step 2.} Proof that $\mathcal{P}$ satisfies all the requirements.\par
We need to construct $g : V' \to T$ and $\psi : V' \dashrightarrow V$. Let $T_0 \subseteq T$ be an affine open dense subset such that $V_t$ is smooth and $\vol(V_t) \leq M$ for any $t \in T_0$. Let $V_0$ be the preimage of $T_0$ and denote by $f_0 : V_0 \to T_0$ the restriction map. After possibly shrinking $T_0$, we may and do assume $f_0 : V_0 \to T_0$ is flat.

Let $\phi : V \dashrightarrow Z$ be the birational map induced by the relative $r_d$-canonical map over $T$ and let $\phi_0 : V_0 \dashrightarrow Z_0$ be the restriction of $\phi$ to $V_0$ where $Z_0$ is the image. After shrinking $T_0$, we may and do assume that, for any $t \in T_0$, $\phi_0|_{V_t}$ is birational onto its image. Then we can embed $Z_0$ into $\bP_{T_0}^m = T_0 \times \bP^m$, where $m := h^0(r_dK_X) - 1$ for a general fiber $X$ of $f_0 : V_0 \to T_0$ and we denote by $h_0 : Z_0 \to T_0$ the projection morphism. Let $A := \mathcal{O}_{\bP^m_{T_0}}(1) |_{Z_0}$ which is a relatively very ample divisor on $Z_0$ over $T_0$.
$$
\xymatrix{
V_0 \ar@{-->}[r]^{\phi_0} \ar[dr]_{f_0}& Z_0 \ar[d]^{h_0}  \ar@{^(->}[r]^i& \bP^m_{T_0} \ar[dl]^\pi \\
 & T_0 &
 }
$$
Then, for a fiber $Y$ of $h_0$, we have 
$$\deg{Y} = (A|_Y)^d \leq \vol(r_dK_X) \leq {r_d}^d M.$$
Since $\deg Y \geq h^0(r_dK_X) - d$, we have $m  \leq N = \rounddown{{r_d}^{d}M} + d$. Thus we have an embedding $Z_0 \to {\bP^N_{T_0}}$ such that the degree of any fiber of $h_0 : Z_0 \to T_0$ is no greater than ${r_d}^d M$.\par

By the representability of the Chow functor, there is a morphism $\rho :T_0 \to S$ such that $Z_0$ is the pull-back of the universal family $\mathscr{X}$. 
%$$\xymatrix{
%Z_0  \ar[r] \ar[d]_{f_0} &  \mathscr{X}  \ar[d]^{p} \\
%T_0   \ar[r]^\rho   & S
%}$$
After possibly shrinking $T_0$ again, we may assume that $\rho(T_0)$ is contained in some $S_i$. Since every blowup of $\mathscr{X}_i$ induces a blow up of $Z_0$, we get a birational model $V'_0$ by blowing up $Z_0$ finitely many times such that the fiber of $f'_0 : V'_0 \to T_0$ is smooth and belongs to $\mathcal{P}$. 
$$\xymatrix{
V'_0	\ar[r]	\ar[d]	\ar@/_2pc/[dd]_{f'_0}&  \mathscr{Y}_i \ar[d]  \\
Z_0  \ar[r] \ar[d]^{f_0} &  \mathscr{X}_i  \ar[d]^{p} \\
T_0   \ar[r]^\rho   &  S_i 
}$$\par
Since $Z_0$ is dense in $Z$, blowing up a center in $Z_0$ also induces a blowup of $Z$ by taking the closure in $Z$. Thus we get a morphism $f' : V' \to T$ such that a general fiber $X'$ of $f'$ is smooth and belongs to $\mathcal{P}$. Finally, after resolving the special fiber of $f'$ and replacing $V'$ with the resulting model, we may and do  assume $V'$ is smooth. Then the morphism $g=f':V' \to T$ is what we want. And $\psi$ is taken to be the inverse of the composition of the two birational maps  $V \dashrightarrow Z \dashrightarrow V'$.\\

{\em Step 3.} Proof of Statement (iii). \par
Let $\alpha : W \to V$ and $\beta : W \to V'$ be a common resolution of $V$ and $V'$. 
Write $K_W = \beta^*K_V + E$ and $K_W = \alpha^*K_{V'} + F$, where $E,F$ are effective. \par
For a general fiber $X'$ of $g: V' \to T$, $\psi|_{X'}$ is birational and let $X$ be the birational transform of $X'$ on $V$. By assumption, there is a $\bQ$-divisor $\Delta \sim_\bQ \delta K_V$ on $V$ such that $X$ is an irreducible non-klt center of $(V,\Delta)$. 
Let $\Delta':=\beta_*(\alpha^*\Delta + \delta E)$, so $\Delta' \sim_\bQ \delta K_{V'}$. \par

Since $X,X'$ are general, they are not contained in the exceptional locus of $p,q$ and $\beta_*{E}$ does not contain $X'$. Thus $\psi$ induces an isomorphism between neighborhoods of the generic points of $X$ and $X'$, then it follows that the log discrepancies over these neighborhoods with respect to $(V,\Delta)$ and $(V', \Delta')$ are the same.
Therefore $X'$ is an irreducible non-klt center of $(V',\Delta')$.
\end{proof}

We will apply the following boundedness theorem of Birkar in our argument. 
\begin{thm}\label{BAB2}
(\cite[Theorem 1.8]{Bir21})
Let $d$ and $r$ be positive integers and $\epsilon$ be a positive real number. Then there exists a positive number $t$, depending only on $d$, $r$ and $\epsilon$ such that the following property holds:
\begin{quote}
Assume that
\begin{itemize}
\item[(1)] $(X,B)$ is a projective $\epsilon$-lc pair of dimension $d$,
\item[(2)] $A$ is a very ample divisor on $X$ with $A^d \leq r$,
\item[(3)] $A - B$ is pseudo-effective,
\item[(4)] $M \geq 0$ is an $\bR$-Cartier $\bR$-divisor with $A-M$ being pseudo-effective.
\end{itemize}
Then one has
$$\lct(X,B,|M|_{\bR}) \geq \lct(X,B,|A|_\bR) \geq t.$$
\end{quote}
\end{thm}

\subsection{An extension theorem}
We need to apply the following useful extension theorem due to Hacon and McKernan. 

\begin{thm}\label{extend} (\cite[Corollary 3.17]{HM06}) 
Let $(V,\Delta),S,C,H$ satisfy the following conditions:
\begin{itemize}
\item[(1)] $(V,\Delta)$ is a log smooth lc $\bQ$-pair;
\item[(2)] $S$ is a component of $\Delta$ with coefficient $1$ such that $(K_V + \Delta)|_S$ is pseudo-effective; 
\item[(3)] $C \geq 0$ is an effective $\bQ$-divisor whose support does not contain $S$; 
\item[(4)] there is a $\bQ$-divisor $G \sim_\bQ K_V + \Delta + C$ such that $G$ does not contain any lc centers of $(V,\roundup{\Delta})$;
\item[(5)] $H$ is a sufficiently ample integral divisor which does not contain $S$ and only depends on $(V,\Delta),S,C$ and $A:=(\dim V+1)H $.
\end{itemize}
We give some notations of divisors on $S$ as follows:  $\Delta_S := (\Delta - S) |_S$, $C_S:= C|_S$, $H_S:=H|_S$, and $A_S :=A|_S$. Then, for any positive integer $m$ such that $m\Delta$ and $mC$ are integral, the image of the restriction map
$$H^0(V,m(K_V + \Delta + C) + H + A) \to H^0(S, m(K_S + \Delta_S + C_S) + H_S + A_S)$$
contains the image of the inclusion
$$i: H^0(S,m(K_S + \Delta_S) + H_S) \to H^0(S, m(K_S + \Delta_S + C_S) + H_S + A_S)$$
induced by the effective divisor $mC_S + A_S$. 
\end{thm}
Note that, in Theorem \ref{extend},  $H$ and $A$ do not depend on $m$.

\section{An extension theorem for fiber spaces--The proof of Theorem \ref{key}}

We start with considering a special case of Theorem \ref{key} where only one general fiber is taken into account. 

\begin{thm}\label{main1}
Let $n$ and $d$ be two integers with $n > d >0$ and $\mathcal{P}$ be a birationally bounded set of smooth projective varieties of dimension $d$. Then there exists a positive number $t$, depending only on $d$ and $\mathcal{P}$, such that the following property holds:
\begin{quote}
Let $f : V \to T$ be a contraction between smooth projective varieties and denote by $X$ a general fiber of $f$, which satisfies the following conditions:
\begin{itemize}
\item[(1)] $V$ is of general type;
\item[(2)] $\dim V = n$, $\dim X = d$ and $X$ is birationally equivalent to an element in the bounded set $\mathcal{P}$;
\item[(3)] there exists a positive rational number $\delta < t$ with $\Delta \sim_\bQ \delta K_V$ such that $X$ is an irreducible non-klt center of $(V,\Delta)$.   
\end{itemize}
Then, for any integer $p$ with $p \geq 2$, the restriction map
$$H^0(V,pK_V) \to H^0(X, pK_X)$$ is surjective.
\end{quote}
\end{thm}

\begin{proof} First of all, by virtue of Lemma \ref{m7}, after possibly replacing $\mathcal{P}$ with a bounded set, we may and do assume that $\mathcal{P}$ is a bounded set and that the general fiber $X$ of $f$ belongs to $\mathcal{P}$. \\

{\em Step 1.} Existence of $t$. \par
By the proof of \cite[Theorem 3.5]{CJ17}, there exists a finite set $\mathcal{S}$ of positive integers depending only on $d$ and $\mathcal{P}$ such that for any $X \in \mathcal{P}$ the canonical ring $\oplus_{i\geq 0}{H^0(iK_X)}$ is generated by finitely many elements whose degree numbers belong to $\mathcal{S}$. Take $l$ be a common multiple of all numbers in $\mathcal{S}$. Then $l$ only depends on $d$ and $\mathcal{P}$. On the other hand, since $X$ is bounded, there is a very ample Cartier divisor $A$ on $X$ such that $A^d = \tilde{v}$ and that $A - K_X$ is ample and effective, where $\tilde{v}$ is a positive number which depends only on $d$ and $\mathcal{P}$.\par
Since $(X,0)$ is canonical, by Theorem \ref{BAB2}, there exists a positive number $t$ depending on $d$, $l$ and $\tilde{v}$ (hence only depending on $d$ and $\mathcal{P}$) such that, for any effective divisor $D$ on $X$ satisfying $A-D$ being pseudo-effective, we have such property that $(X,2tlD)$ is klt. \par
In next steps, we prove that this number $t$ satisfies requirements of the theorem.\\

{\em Step 2.} The setup for blowing-up $X$.\par 
 By Lemma \ref{perturb}, take a small rational number $\epsilon > 0$ with $\delta(1+\epsilon) < t$ and replace $\delta$ with $\delta(1+\epsilon)$, $\Delta$ with another $\Delta'\sim_\bQ \delta K_V$. Then we may and do assume that $X$ is a pure and exceptional lc center of $(V,\Delta)$.\par
Let $\nu: W \to V$ be a log resolution of $(V, \Delta)$ and $E$ be the unique lc place on $W$ with $\nu(E)=X$. By the existence of the minimal model \cite{BCHM10} and Lemma \ref{m2}(5), after possibly replacing $W$ with a further resolution, we may and do assume that $P_\sigma(\nu^*(K_V))$ is nef and that the components of $N_\sigma(\nu^*(K_V))$, exceptional divisors of $\nu$ and the birational transform of $\Delta$ have simple normal crossings. Denote by $g:E \to X$ the restriction map of $\nu$ to $E$. We see $g$ is a contraction by the connectedness lemma \cite[Theorem 17.4]{F92}. Then write   
$$\nu^*(K_V + \Delta) = K_W + E + \Gamma_1 + \Gamma_2 - F,$$
such that $\Gamma_1,\Gamma_2,F$ are effective with no common components, that $\Gamma_{1,E}$ is $g$-horizontal and that $\Gamma_{2,E}$ is $g$-vertical, where $\Gamma_{1,E},\Gamma_{2,E},F_E$ are the restriction $\Gamma_1|_E, \Gamma_2|_E,F|_E$, respectively. Since $E$ is the unique lc place of $(V,\Delta)$ with $\nu(E)=X$, we see that $\Gamma_{1,E} < 1$. In particular, there is no $g$-horizontal non-klt centers of $(E, \Gamma_{1,E} + \Gamma_{2,E} - F_E)$.\\

{\em Step 3.} The pair $(E, \Gamma_{1,E} + \Gamma_{2,E} - F_E)$ is sub-klt. \par
In fact, we would like to prove that, for any effective $\bR$-divisor $G \sim_\bR A$ on $X$, $(E, \Gamma_{1,E} + \Gamma_{2,E} - F_E + 2\delta(l-1)g^*G)$ is sub-klt.  \par
Since $g:(E,\Gamma_{1,E} + \Gamma_{2,E} - F_E) \to X$ is a klt-trivial fibration (see Subsection \ref{adjunction}(iv)$\sim$(vii)), we can do adjunction for $g:(E,\Gamma_{1,E} + \Gamma_{2,E} - F_E) \to X$. Write 
$$K_E + \Gamma_{1,E} + \Gamma_{2,E} - F_E \sim_\bR g^*(K_X + B_X + M_X),$$
where $B_X$ is the discriminant part and $M_X$ is the moduli part, and we have $B_X + M_X \sim_\bQ \Delta|_X \sim_\bQ \delta K_X$ by direct computations. We have $B_X \geq 0$ by \cite[Lemma 4.2(3)]{Amb99} (see also Subsection \ref{adjunction}(viii)). \par
%Then, we will pick $E'$, a birational model of $E$ as follows. First, since $(E, \Gamma_{1,E} + \Gamma_{2,E} - F_E)$ is sub-klt, by \ref{}, we can take a birational model $E'' \to E$, such that if we let $K_{E''} + \Gamma_{1,E''} + \Gamma_{2,E''} - F_{E''} = \rho^*(K_E + \Gamma_{1,E} + \Gamma_{2,E} - F_E)$ be the crepant pullback, where $\Gamma_{1,E''},\Gamma_{2,E''},F_{E''}$ are effective with no common components, $\Gamma_{1,E''}$ is $g'$-horizontal, and $\Gamma_{2,E''}$ is $g'$-vertical.

We are prepared to use Lemma \ref{m5} for the contraction $g:E\to X$. By Lemma \ref{m4.5}, we can take an excellent base extension for $g:(E,\Gamma_{1,E} + \Gamma_{2,E} - F_E) \to X$. Let $\pi : E' \to E$ be a birational morphism such that $E'$ is smooth and that $g' :E' \dashrightarrow X'$ is an induced contraction from $g$.  
$$
\xymatrix{
E' \ar[r]^{\pi} \ar[d]_{g'} & E \ar[d]^g \\
X' \ar[r]^{\sigma} & X
}$$
Let $K_{E'} + \Gamma_{1,E'} + \Gamma_{2,E'} - F_{E'} = \pi^*(K_E + \Gamma_{1,E} + \Gamma_{2,E} - F_E)$ be the crepant pullback, where $\Gamma_{1,E'},\Gamma_{2,E'},F_{E'}$ are effective with no common components, $\Gamma_{1,E'}$ is $g'$-horizontal and $\Gamma_{2,E'}$ is $g'$-vertical. Let $B_{X'}$, $M_{X'}$ be the respective discriminant part, the moduli part due to adjunction for $g' : (E', \Gamma_{1,E'} + \Gamma_{2,E'} - F_{E'}) \to X'$. Here we require that $\sigma_*M_{X'} = M_X$ as explained in Subsection \ref{adjunction}(ii). If we do adjunction for $g': (E', \Gamma_{1,E'} + \Gamma_{2,E'} - F_{E'} + 2\delta(l-1)\pi^*g^*G) \to X'$, we will get the discriminant part $B_{X'} + 2\delta(l-1)\sigma^*(G)$ and the moduli part $M_{X'}$ (see \ref{adjunction}(ii)). Thus $X'$ is still an excellent base extension as well for the klt-trivial fibration $g: (E, \Gamma_{1,E} + \Gamma_{2,E} - F_E + 2\delta(l-1)g^*G) \to X$. By Lemma \ref{m5}, it suffices to prove that the sub-pair $(X', B_{X'} + 2\delta(l-1)\sigma^*G)$ is sub-klt.\par

Since $M_{X'}$ is nef and $\sigma^*(\delta A)$ is nef and big, there exists an effective $\bR$-divisor $D_{X'} \sim_\bR \sigma^*(\delta A) + M_{X'}$. It is sufficient to prove that the pair
$$(X', B_{X'} + 2\delta(l-1)\sigma^*G + D_{X'})$$ is sub-klt.\par 

Since $$K_{X'} + B_{X'} + 2\delta(l-1)\sigma^*G+  D_{X'} \overset{\sigma}\sim_\bR 0\ \ (\text{over}\ X),$$ let $D_X := \sigma_*(D_{X'}) \sim_{\bR} \delta A + M_X$, then we have 
$$K_{X'} + B_{X'} + 2\delta(l-1)\sigma^*G + D_{X'} = \sigma^*(K_X + B_X + 2\delta(l-1)G + D_X).$$
So it suffices to prove that $(X,B_X + 2\delta(l-1)G + D_X)$ is klt. Since 
\begin{align*}
0 \leq B_X + 2\delta(l-1)G + D_X &\sim_\bR B_X + 2\delta(l-1)A + \delta A +  M_X \\
&\sim_\bR 2\delta(l-1)A + \delta A + \delta K_X \\
&\leq 2\delta l A
\end{align*} and $\delta < t$, we see that $(X,B_X + 2\delta(l-1)G + D_X)$ is klt by the construction of $t$ in Step 1. \\

{\em Step 4.} The pseudo-effectiveness of $g^*(\delta K_X) - \Gamma_{2,E}$.\par

Since $(E, \Gamma_{1,E} + \Gamma_{2,E} - F_E)$ is sub-klt, by Lemma \ref{m1.5}, there is a log resolution $\rho : \hat{E} \to E$ from a smooth model $\hat{E}$ such that if we write $$K_{\hat{E}} + \Gamma_{1,\hat{E}} + \Gamma_{2,\hat{E}} - F_{\hat{E}} = \rho^*(K_E + \Gamma_{1,E} + \Gamma_{2,E} - F_E)$$ for the crepant pull-back, where $\Gamma_{1,\hat{E}}$, $\Gamma_{2,\hat{E}}$ , $F_{\hat{E}}$ are effective with no common components, $\Gamma_{1,\hat{E}}$ is horizontal over $X$ and $\Gamma_{2,\hat{E}}$ is vertical over $X$, then $(\hat{E}, \Gamma_{1,\hat{E}} + \Gamma_{2,\hat{E}} - F_{\hat{E}})$ is terminal. Let $\hat{g} =g \circ \rho : \hat{E} \to X$ be the composite map. 

Then take $\tau : X'' \to X$ to be a birational morphism from a smooth model $X''$ that flattens $\hat{g} : \hat{E} \to X$. Let $\xi : E'' \to \hat{E}$ be a birational morphism from a smooth model $E''$ (which factors through the flat model $U\to X''$ of $\hat{g}$) such that the induced birational map $g'' :E'' \dashrightarrow X''$ is a morphism. Let $\eta = \rho \circ \xi : E'' \to E$ be the composition.  
$$
\xymatrix{
E'' \ar@/^1.5pc/[rr]^\eta \ar[r]^{\xi} \ar[d]_{g''} & \hat{E} \ar[r]^\rho \ar[d]^{\hat{g}} & E \ar[dl]^g  \\
X'' \ar[r]^{\tau} & X
}$$
Let 
$$K_{E''} + \Gamma_{1,E''} + \Gamma_{2,E''} - F_{E''} = \xi^*(K_{\hat{E}} + \Gamma_{1,\hat{E}} + \Gamma_{2,\hat{E}} - F_{\hat{E}})$$
be the crepant pull-back on $E''$, where $\Gamma_{1,{E''}}$, $\Gamma_{2,{E''}}$ and $F_{{E''}}$ are effective with no common components and $\Gamma_{1,{E''}}$ is horizontal over $X$, and $\Gamma_{2,{E''}}$ is vertical over $X$.
Since $(\hat{E}, \Gamma_{1,\hat{E}} + \Gamma_{2,\hat{E}} - F_{\hat{E}})$ is terminal, $\Gamma_{2,E''}$ is just the birational transform of $\Gamma_{2,{\hat{E}}}$. In particular, any component in $\Gamma_{1,\hat{E}}$ and $\Gamma_{2,\hat{E}}$ can not be $\xi$-exceptional. In other words, the image of any prime component of $\Gamma_{2,E''}$ under $g''$ is still a prime divisor on $X''$, since $\tau$ flattens $\hat{g} : \hat{E} \to X$.  We may also require that $X''\to X$ is an excellent base extension with respect to the pair $(\hat{E}, \Gamma_{1,\hat{E}} + \Gamma_{2,\hat{E}} - F_{\hat{E}})$. 

Let $B_X$ and $M_X$ be as in Step 3. Let $B_{X''}$ and $M_{X''}$ be the discriminant part and the moduli part of the adjunction for $g'' : (E'', \Gamma_{1,E''} + \Gamma_{2,E''} - F_{E''}) \to X''$. We write $B_{X''} = B_{X''}^+ - B_{X''}^-$ in the unique way with $B_{X''}^+$ and $B_{X''}^-$ being effective with no common components. Note that $M_{X''}$ is nef in our situation. \par

Since $\eta_*(\Gamma_{2,E''}) = \Gamma_{2,E}$, it is sufficient to prove that $${g''}^*\tau^*(\delta K_X) - \Gamma_{2,E''}$$ is pseudo-effective. Since 
$${g''}^*\tau^*(\delta K_X) - \Gamma_{2,E''} = {g''}^*(B_{X''}^+) - \Gamma_{2,E''} + {g''}^*(\tau^{*}(\delta K_X) - B_{X''}^+),$$
it is enough to prove that both ${g''}^*(B_{X''}^+) - \Gamma_{2,E''}$ and $\tau^{*}(\delta K_X) - B_{X''}^+$ are pseudo-effective.\par

We prove that ${g''}^*(B_{X''}^+) - \Gamma_{2,E''} \geq 0$. In fact,  for any component $D''$ of $\Gamma_{2,E''}$ on $E''$, let $P'' := g''(D'')$. By the above assumption, $P''$ is a prime divisor on $X''$. By the definition of the discriminant part, we have $\mu_{P''}({B_{X''}}) = 1-t$, where 
$$t = \sup\{a|\Gamma_{1,E''} + \Gamma_{2,E''} - F_{E''} + a g''^*(P'') \text{ is lc over the generic point of $P''$}\}.$$

Let $c:= \mu_{D''}(g''^*(P''))$ and we see $c$ is a positive integer. Then we have $\mu_{D''}(\Gamma_{2,E''} + t g''^*(P''))  = \mu_{D''}(\Gamma_{2,E''}) + tc \leq 1$. Thus, 
\begin{align*}
\mu_{D''}g''^*(B_{X''}) &= c \cdot \mu_{P''}(B_{X''}) \\
&= c(1-t) \\
&\geq 1-ct \geq \mu_{D''}(\Gamma_{2,E''}).
\end{align*}
Thus $g''^*(B_{X''}^+) - \Gamma_{2,E''} \geq_{\bR} 0$ holds.\par

Then we are left to prove that $\tau^{*}(\delta K_X) - B_{X''}^+$ is pseudo-effective.
We rewrite
$$K_{X''} + B_{X''} + M_{X''} = \tau^*(K_X + B_X + M_X),$$
as 
\begin{equation}
\tau^*(B_X + M_X) - M_{X''} = K_{X''/X} + B_{X''}^+ - B_{X''}^-. \label{eq3.1}
\end{equation}
Note that $M_{X''}$ is actually determined up to $\bR$-linear equivalence, we may choose from the very beginning that $\tau_*M_{X''}=M_X$. Since $M_{X''}$ is nef, $\tau^*(M_X) \geq M_{X''}$ by the negativity lemma \cite[Lemma 3.3]{Bir12}. Since $B_X \geq 0 $, the left hand side of \eqref{eq3.1} is effective. Since $K_{X''/X} \geq 0$ by the smoothness of $X$ and that $B_{X''}^+,B_{X''}^-$ have no common components, we get 
$$\tau^*(B_X + M_X) - M_{X''} \geq B_{X''}^+.$$ 
Thus $\tau^*(\delta K_X) - B_{X''}^+ \sim_\bR \tau^*(B_X + M_X) - B_{X''}^+ \geq M_{X''}$ is pseudo-effective. \\

{\em Step 5.} A key application of Theorem \ref{extend} (yielding Inequality \eqref{step5}). \par
By Step 4, $g^*(\delta K_X) - \Gamma_{2,E}$ is pseudo-effective. Noting that $$g^*(\delta K_X) - \Gamma_{2,E} \sim_\bQ K_E - g^*K_X + \Gamma_{1,E} - F_E,$$ 
there exists an ample divisor $H_E$ such that for any sufficiently large and divisible integer $m$, we have $h^0(E,m(K_E - g^*K_X + \Gamma_{1,E} - F_E) + H_E) > 0$ by \cite[Corollary 11.2.13]{Laz2}. After taking a sufficiently ample divisor $H$ on $W$ such that $H|_E - H_E \geq 0$ and replacing $H_E$ with $H|_E$, we can assume $H|_E = H_E$. Note that $H$ does not depend on $m$.\par
For any $D_m \in |m(K_E - g^*K_X + \Gamma_{1,E} - F_E) + H_E|$, there is an inclusion of linear system 
$$|g^*(mK_X)| \hookrightarrow |m(K_E + \Gamma_{1,E}) + H_E|,$$
induced by the divisor $D_m + mF_E$.\par
To use Theorem \ref{extend}, we need to check all the conditions. Actually the notation $(V,\Delta),S,C,H$ in Theorem \ref{extend}, respectively, corresponds to $(W,E + \Gamma_1),E,\Gamma_2,H$ here. Besides, the divisor $A$ in Theorem \ref{extend} should be $(\dim W+1)H=(n+1)H$.
Note that $$K_W + E + \Gamma_1 + \Gamma_2 \sim_\bQ \nu^*(K_V + \Delta) + F.$$ 
Since $X$ is a general fiber of $f : V\to T$, that $F$ has no common components with $E$ and $\Gamma_{1}$, and that ${\rm Supp }(F) \cup {\rm Supp }(E) \cup {\rm Supp } (\Gamma_{1})$ has simple normal crossings, we see there exists an effective $\bQ$-divisor $\Lambda \sim_\bQ K_W + E + \Gamma_{1} + \Gamma_{2}$ such that $\Lambda$ does not contain any lc center of $(W,E + \roundup{\Gamma_1})$. In a word, all conditions of Theorem \ref{extend} in our setting here are satisfied. \par

Then,  by Theorem \ref{extend}, we see that any element of the image of
$$i: |m(K_E + \Gamma_{1,E}) + H_E| \to |m(K_E + \Gamma_{1,E} + \Gamma_{2,E}) + (n+2)H_E|$$
can be lifted to an element in
$$|m(K_W + E + \Gamma_1 + \Gamma_2) +  (n+2)H|=|m(\nu^*((1+\delta) K_V)+F) + (n+2)H|,$$
where $i$ is the injection induced by the divisor $m\Gamma_{2,E} + (n+1)H_E$ and $m$ is any positive integer making all above divisors integral.\par
By the argument above, for any effective member $N_m \in |g^*(mK_X)|$, there exists an effective $\Theta_m \in |m(\nu^*((1+\delta) K_V)+F) + (n+2)H|$, such that 
$$\Theta_m |_E = N_m + (D_m + mF_E) + (m\Gamma_{2,E} + (n+1)H_E).$$
Denote by $F_m$ the fixed part of $|m(\nu^*((1+\delta) K_V)+F) + (n+2)H|$ and by $Z_m$ the fixed part of $|g^*(mK_X)|$. Then, by Lemma \ref{m1}, we get
 \begin{equation*}F_m|_E \leq Z_m + D_m + mF_E + m\Gamma_{2,E} + (n+1)H_E
 \end{equation*}
for any effective member $D_m \in | m (g^*(\delta K_X) - \Gamma_{2,E}) + H_E| $.\par
Let $Y_m := \fix{m (g^*(\delta K_X) - \Gamma_{2,E}) + H_E}$. Since $F_m$ does not depend on the choice of $D_m$, we have 
\begin{equation}F_m|_E \leq Z_m + Y_m + mF_E + m\Gamma_{2,E} + (n+1)H_E.\label{fix-ineq}\end{equation}

Divide Inequality \eqref{fix-ineq} by $m$ on both sides and let $m$ tends to the infinity. Since $H$ and $H_E$ are sufficiently ample,  while applying Lemma \ref{m2}, we have 
\begin{eqnarray*}\lim_{m \to \infty}F_m/m &=& N_\sigma{(\nu^*((1+\delta)K_V) +F)} = (1+\delta)N_\sigma{(\nu^*(K_V)) +F},\\
\lim_{m \to \infty}{Z_m/m}& = &N_\sigma{(g^*(K_X))},\\
  \lim_{m \to \infty}{Y_m/m}& = &N_\sigma{(g^*(\delta K_X) - \Gamma_{2,E})}.
 \end{eqnarray*}
 Hence Inequality \eqref{fix-ineq} reads
\begin{equation}(1+\delta)N_{\sigma}(\nu^*(K_V))|_E \leq N_\sigma(g^*(K_X))+ Q, \label{step5}\end{equation}
where $Q := N_\sigma{(g^*(\delta K_X) - \Gamma_{2,E})} + \Gamma_{2,E}$.  We see that $Q \geq 0$. By the argument in Step 4, we have seen that 
$$g^*(\delta A) - Q=(g^*(\delta A)-\Gamma_{2,E})-N_\sigma{(g^*(\delta K_X) - \Gamma_{2,E})}$$
is pseudo-effective, where $A$ is the very ample divisor constructed in Step 1.\\

{\em Step 6.} The surjective map \eqref{step6} by Kawamata-Viehweg vanishing theorem. \par

Fix an integer $p$ with $2\leq p \leq l$ and define 
\begin{align*}
L &:= \rounddown{\nu^*(K_V + \Delta) - K_W - E + (p-1- \delta)N_\sigma(\nu^*K_V)} \\
 &= \rounddown{\Gamma_1 + \Gamma_2 - F + (p-1- \delta)N_\sigma(\nu^*K_V)}. 
\end{align*}
Note that, since $N_\sigma(\nu^*K_V)$ does not contain $E$, the support of $L$ does not contain $E$. \par
Since 
\begin{eqnarray*}
&& \nu^*(pK_V) - E - K_W - L \\
&\sim_\bQ& (p-1-\delta)\nu^*K_V + \Gamma_1 + \Gamma_2 - F - L  \\
&\sim_\bQ& \langle\Gamma_1 + \Gamma_2 - F + (p-1-\delta)N_\sigma(\nu^*K_V)\rangle + (p-1-\delta)P_\sigma(\nu^*K_V),
\end{eqnarray*}
where $(W,\fraction{\Gamma_1 + \Gamma_2 - F + (p-1-\delta)N_\sigma(\nu^*K_V)})$ is a klt pair and $(p-1-\delta)P_\sigma(\nu^*K_V)$ is nef and big, the restriction map
\begin{equation}H^0(W, \nu^*(pK_V) - L) \longrightarrow H^0(E, g^*(pK_X) - L|_E)
\label{step6}\end{equation}
is surjective, by the Kawamata-Viehweg vanishing theorem.\\

{\em Step 7.} The injective map \eqref{step7} between cohomological groups.\par

We explore the divisor $L|_E$.
By \eqref{step5}, we have
\begin{eqnarray*}
L|_E &=& \rounddown{\Gamma_{1,E} + \Gamma_{2,E} - F_E + (p-1-\delta)N_\sigma(\nu^*(K_V)|_E} \\
&\leq& \rounddown{\Gamma_{1,E} + \Gamma_{2,E} - F_E + \frac{p-1-\delta}{1+\delta}(N_\sigma(g^*(K_X)) + Q)} \\
&\leq& \roundup{L_{1,E}} + \rounddown{L_{2,E}},
\end{eqnarray*}
where $L_{1,E} := \frac{p-1-\delta}{1+\delta} N_\sigma(g^*(K_X))$ and $L_{2,E} := \Gamma_{1,E} + \Gamma_{2,E} - F_E + \frac{p-1-\delta}{1+\delta}Q$. 
Since $g^*(\delta A) - Q$ is pseudo-effective, by Lemma \ref{m4}, there exists an effective $G \sim_\bR A$ on $X$ such that $g^*(2\delta G) - Q \geq 0$. Thus $L_{2,E} \leq \Gamma_{1,E} + \Gamma_{2,E} - F_E + (p-1)g^*(2\delta G)$.\par
In Step 3, we have shown that $(E, \Gamma_{1,E} + \Gamma_{2,E} - F_E + (p-1)g^*(2\delta G))$ is sub-klt. Hence $(E, L_{2,E})$ is also sub-klt. It follows that $\rounddown{L_{2,E}} \leq 0$. Also it is clear that one has  $L_{1,E} \leq N_\sigma(g^*(pK_X))$.\par
Thus
\begin{eqnarray*}
H^0(E, g^*(pK_X) - L|_E) &\supseteq& H^0(E,g^*(pK_X) - \roundup{L_{1,E}} - \rounddown{L_{2,E}}) \\
&\supseteq& H^0(E,g^*(pK_X) - \roundup{N_\sigma(g^*(pK_X))}) \\
&=& H^0(E,g^*(pK_X)).
\end{eqnarray*} 
As the conclusion, there is the following injective map 
\begin{equation}j : H^0(E,g^*(pK_X)) \hookrightarrow H^0(E, g^*(pK_X) - L|_E).\label{step7}
\end{equation}

{\em Step 8.} Concluding the statement of this theorem.\par
In the following commutative diagram, the first row is surjective by \eqref{step6} and the injectivity of $j: H^0(g^*(pK_X)) \to H^0(g^*(pK_X) - L|_E)$ is due to \eqref{step7}. Denote by $r$ the restriction map $$H^0(\nu^*(pK_V)) \to H^0(E, g^*(pK_X)).$$ We also see $H^0(\nu^*(pK_V)) = H^0(\nu^*(pK_V) + \roundup{F})$ since the support of $F$ is $\nu$-exceptional. Besides, the two vertical injections are due to the fact $-L \leq \roundup{F}$ and $-L|_E \leq \roundup{F|_E}$. \par
$$\xymatrix{
H^0(\nu^*(pK_V) - L ) \ar@{^(->}[d] \ar@{->>}[r]
				& H^0(g^*(pK_X) - L|_E) \ar@{^(->}[d]\\
H^0(\nu^*(pK_V) + \roundup{F}) \ar@{=}[d] \ar[r]
				& H^0(g^*(pK_X) + \roundup{F|_E}) \\		
H^0(\nu^*(pK_V)) \ar@{=}[d] \ar[r] ^r 
				& H^0(g^*(pK_X)) \ar@{=}[d] \ar@{^(->}[u] \ar@/_7.5pc/[uu]^{j} \\
H^0(pK_V) \ar[r] &
		H^0(pK_X)		
}$$ \par
Now it can be seen that $r$ is surjective by chasing the diagram. In fact,  for any element in $H^0(g^*(pK_X))$, it is embedded into $H^0(g^*(pK_X) - L|_E)$ by $j$ and then is lifted to an element in the group $H^0(\nu^*(pK_V) - L)$. Thus we naturally get an element in $H^0(\nu^*(pK_V))$, which is the pre-image we need to find.\par
Finally, we have already proved that, for any integer $p \in [2,l]$, the restriction map $H^0(pK_V) \to H^0(pK_X)$ is surjective. Since for $i \geq l+1$, any element in $H^0(X,iK_X)$ can be written as a sum of products of elements in $H^0(X,pK_X)$ with $2 \leq p \leq l$ by the definition of $l$, we see that $H^0(V,iK_V) \to H^0(X,iK_X)$ is surjective as well for any integer $i \geq 2$. We have proved the theorem. 
\end{proof}

\begin{cor}\label{main2} (=Theorem \ref{key})
Under the same condition as that of Theorem \ref{main1}, pick two general fibers $X_1$ and $X_2$ of $f:V\to T$. Then, for any integer $p \geq 2$, the restriction map 
$$H^0(V,pK_V) \to H^0(X_1, pK_{X_1}) \oplus H^0(X_2,pK_{X_2})$$
is surjective. In particular, $r_s(V)\leq \text{max}\{r_s(X),2\}$ where $X$ is the general fiber of $f$. 
\end{cor}
\begin{proof} The spirit of the proof is similar to that of Theorem \ref{main1}. So there is no need to repeat the same argument except for the following points. \par
Replacing $t$ with $\frac{1}{2}t$, we may assume that there exists a $\bQ$-divisor $\Delta \sim_\bQ \delta K_V$ for a positive rational $\delta < 2t$ such that both $X_1$ and $X_2$ are irreducible pure and exceptional lc centers of $(X, \Delta)$. Let $\nu: W \to V$ be a log resolution and $E_1$, $E_2$ be lc places of $X_1$, $X_2$, respectively and denote by $g_i : E_i \to X_i$ for $i=1,2$. Let $$L:= \rounddown{\nu^*(K_V + \Delta) - K_W - E_1 - E_2 + (p-1- \delta)N_\sigma(\nu^*K_V)}.$$
Going on the similar argument as in the proof of Theorem \ref{main1}, we will get the surjective map 
$$H^0(W, \nu^*(pK_V) - L) \to H^0(E_1, g_1^*(pK_X) - L|_{E_1}) \oplus H^0(E_2, g_2^*(pK_X) - L|_{E_2}).$$
Then we explore $L_{E_1}$ and $L_{E_2}$, respectively, in the similar way.  All other argument follows similarly to that in Theorem \ref{main1} without any problems. So we leave the details to interested readers. 
\end{proof}

\section{The proof of Theorem \ref{C-L-1} and Corollary \ref{C-L-2}}

Let us start with recalling the following definition.

\begin{defi}(\cite[Definition 6.5]{CJ17})
    Given a birationally bounded set $\mathscr{X}$ of smooth projective varieties and given a positive number $c$, we say that a fibration $f : X \to T$ between smooth projective vareities satisfies condition $(B)_{\mathscr{X},c}$ if
    \begin{enumerate}
        \item a general fiber $F$ of $f$ is birationally equivalent to an element of $\mathscr{X}$;
        \item for a general point $t \in T$, there exists an effective $\bQ$-divisor $D_t$ with $D_t \sim_\bQ \epsilon K_X$ for a positive rational number $\epsilon < c$, such that the fiber $F_t = f^{-1}(t)$ is an irreducible non-klt center of $(X,D_t)$.
    \end{enumerate}
\end{defi}

The next theorem reduces our problem to prove the extension theorem in the style of Theorem \ref{key}.  

\begin{thm}(\cite[Theorem 6.8]{CJ17})\label{fibration} Let $n>1$ be an integer. Fix a function $\lambda: \bZ_{>0} \times \bZ_{>0} \to \bR_{>0}$. There exists a constant $\mathcal{K} > 0$ and $n-1$ integers $M_{n-1} > M_{n-2} > \cdots > M_1 > 0$ such that, for any smooth projective $n$-fold $X$ with $\vol(X) \geq \mathcal{K}$, the pluricanonical map $\varphi_{a,X}$ is birational for all integers $a \geq 2$, unless that, after birational modifications, $X$ admits a fibration $f:X \to Z$ which satisfies Condition $(B)_{\mathscr{X}_{k,M_k^k},\lambda(k,M_k^k)}$, where $\mathscr{X}_{k,M_k^k}$ is the set of smooth projective $k$-folds $X$ of general type with $\vol(K_X) \leq M_k^k$.
\end{thm}

Now we can prove our main theorem as follows.

\begin{proof}[Proof of Theorem \ref{C-L-1}]
By Theorem \ref{key}, for any integer $d$ and a birationally bounded set $\mathcal{P}$ consisting of smooth projective varieties of dimension $d$, there is a positive number $t = t(d,\mathcal{P})$ satisfying the conditions in Theorem \ref{key} (see Step 1 in the proof of Theorem \ref{main1}). Then we define the following function
\begin{align*}
    \lambda : \bZ_{>0} \times \bZ_{>0} &\to \bR_{>0} \\
    (d,M) &\mapsto t(d,\mathscr{X}_{d,M}), 
\end{align*}
where $\mathscr{X}_{d,M}$ is the set of smooth projective $d$-folds of general type with the canonical volume $\leq M$.
By Theorem \ref{fibration}, there exists a constant $\mathcal{K}>0$ and $n-1$ positive integers $M_{n-1}, \cdots M_1$ such that, for any smooth projective $n$-fold $V$ with $\vol(V) \geq \mathcal{K}$, the $m$-pluricanonical map $\varphi_{m,V}$ is birational for all integers $m \geq 2$, unless that, after replacing $V$ with a birational modification, there is a fibration $f : V \to T$ which satisfies Property $(B)_{\mathscr{X}_{k,M_k^k},\lambda(k,M_k^k)}$ for some positive integer $k\leq n-1$. Then, by Theorem \ref{key} while recalling the definition of $\lambda$, the pluricanonical map $\varphi_{m,V}$ is birational for any $m \geq r_s(X_t)$ where $X_t$ is the general fiber of $f$. Noting that $r_s(X_t)\leq r_k$ and that the sequence $\{r_k\}$ is increasing, we have proved that the $m$-canonical map $\varphi_{m,V}$ is birational for all $m \geq r_{n-1}$. This completes the proof of the first statement. 

We consider a variety with positive geometric genus. On one hand, the canonical volume is linearly increasing with regard to the geometric genus due to the existence of Noether inequality (see Lemma \ref{Noether}). On the other hand, given a fibration $f:V\to T$ with $p_g(V)>0$, each general fiber $X_t$ of $f$ still has $p_g(X_t)>0$. Hence one just simply replace $r_k$ with $r_k^+$ and copy the same proof in the first part. So the second statement follows. 
\end{proof}

Before proving Corollary \ref{C-L-2}, we need to apply the following effective inequality.

\begin{thm}\label{NN} (\cite[Theorem 1.1]{CJ23}) Fix two integers $n$ and $k$ with $n>0$ and $1\leq k< n$. There exist positive numbers $a_{n,k}$ and $b_{n,k}$ such that the following inequality 
$$\vol(X)\geq a_{n,k}h^0(X,\Omega_X^k)-b_{n,k}$$
holds for every smooth projective $n$-fold $X$ of general type. 
\end{thm}

\begin{thm} (=Corollary \ref{C-L-2}) 
     For any integer $n \geq 2$, there exists a number $M(n) > 0$.  For any smooth projective $n$-fold $V$ with either $|\chi(\mathcal{O}_V)| > M(n)$ or 
     $h^i(\OO_V)>M(n)$ for some positive integer $i\leq n$, the pluricanoncial map $\varphi_{m,X}$ is birational for all integers $m \geq r_{n-1}$.
\end{thm}
\begin{proof}
    Note that $\chi(\mathcal{O}_V) = \sum_{i=0}^n{(-1)^i h^i(\mathcal{O}_V)}$. If $|\chi(\mathcal{O}_V)|$ is sufficiently large, then so is $h^i(\mathcal{O}_V)$ for some positive number $i \leq n$. By Theorem \ref{NN}, $\vol(V)$ is also sufficiently large. Thus there exists a number $M(n) > 0$ such that, whenever $|\chi(\mathcal{O}_V)|>M(n)$    or $h^i(\OO_V)>M(n)$, 
    $\vol(V) > \mathfrak{V}(n)$  where $\mathfrak{V}(n)$ is the same constant described in Theorem \ref{C-L-1}. Hence the statement follows directly from Theorem \ref{C-L-1}. 
\end{proof}

\section*{\bf Acknowledgments}
The authors appreciate fruitful discussions with Zhi Jiang and Chen Jiang during the preparation of this paper. Especially Chen Jiang pointed out to us an imprecise application of the adjunction in Subsection \ref{adjunction} in an earlier version of this paper. The second author would like to thank Jianshi Yan, Yu Zou, Minzhe Zhu, Mengchu Li, Wentao Chang for useful discussions and help in study.

\end{document}